\documentclass{daj}

\usepackage[latin1]{inputenc}
\usepackage{t1enc}
\usepackage{amsmath, amssymb, amsfonts, amsthm, amsopn,epsfig}
\usepackage[none]{hyphenat}
\usepackage{tikz}
\usepackage{float}
\usepackage{graphicx}
\usepackage{caption}
\usepackage[margin=1in]{geometry}
\usepackage[toc,page]{appendix}
\usepackage{epstopdf}
\usepackage{etoolbox}
\usepackage[colorlinks=true]{hyperref}
\usepackage{dsfont}
\usetikzlibrary{matrix, positioning,decorations,calligraphy}

\theoremstyle{plain}
\newtheorem{theorem}{Theorem}

\newtheorem{claim}[theorem]{Claim}
\newtheorem{lemma}[theorem]{Lemma}

\theoremstyle{definition}

\dajAUTHORdetails{%
  title = {Small Doubling, Atomic Structure and $\ell$-Divisible Set Families}, 
  author = {Lior Gishboliner, Benny Sudakov and Istv\'an Tomon},
  plaintextauthor = {Lior Gishboliner, Benny Sudakov, Istv\'an Tomon},
  plaintexttitle = {Small Doubling, Atomic Structure and $\ell$-Divisible Set Families}, 
  keywords = {Eventown, $\ell$-divisible set family, Frankl-Odlyzko conjecture},
}   

\dajEDITORdetails{%
   year={2022},
   number={11},
   received={17 August 2021},   
   published={30 September 2022},  
   doi={10.19086/da.38586},       
}   

\begin{document}

\begin{frontmatter}[classification=text]

\title{Small Doubling, Atomic Structure and $\ell$-Divisible Set Families} 

\author[lg]{Lior Gishboliner}
\author[bs]{Benny Sudakov}
\author[it]{Istv\'an Tomon}

\begin{abstract}
Let $\mathcal{F}\subset 2^{[n]}$ be a  set family such that the intersection of any two members of $\mathcal{F}$ has size divisible by $\ell$. 
		The famous Eventown theorem states that if $\ell=2$ then $|\mathcal{F}|\leq 2^{\lfloor n/2\rfloor}$, and this bound can be achieved by, e.g., an `atomic' construction, i.e. splitting the ground set into disjoint pairs and taking their arbitrary unions. Similarly, splitting the ground set into disjoint sets of size $\ell$ gives a family with pairwise intersections divisible by $\ell$ and size $2^{\lfloor n/\ell\rfloor}$. 
		Yet, as was shown by Frankl and Odlyzko, these families are far from maximal. For infinitely many $\ell$, they constructed families $\mathcal{F}$ as above of size $2^{\Omega(n\log \ell/\ell)}$. On the other hand, if the intersection of \emph{any number} of sets in $\mathcal{F}\subset 2^{[n]}$ has size divisible by $\ell$, then it is easy to show that $|\mathcal{F}|\leq 2^{\lfloor n/\ell\rfloor}$. In 1983 Frankl and Odlyzko conjectured that $|\mathcal{F}|\leq 2^{(1+o(1)) n/\ell}$ holds already if one only requires that for some $k=k(\ell)$ any $k$ distinct members of $\mathcal{F}$ have an intersection of size divisible by $\ell$.
		We completely resolve this old conjecture in a strong form, showing that $|\mathcal{F}|\leq 2^{\lfloor n/\ell\rfloor}+O(1)$ if $k$ is chosen appropriately, and the $O(1)$ 
		error term is not needed if (and only if) $\ell \, | \, n$, and $n$ is sufficiently large. Moreover the only extremal configurations have `atomic' structure as above.
		Our main tool, which might be of independent interest, is a structure theorem for set systems with small `doubling'.
\end{abstract}
\end{frontmatter}


\section{Introduction}
	
	An \emph{eventown} is a family $\mathcal{F}\subset 2^{[n]}$ such that $|A\cap B|$ is even for any $A,B\in\mathcal{F}$. The famous Eventown theorem of Berlekamp \cite{B69},  also proved independently by Graver \cite{G75}, states that if $\mathcal{F}\subset 2^{[n]}$ is an eventown, then $|\mathcal{F}|\leq 2^{\lfloor n/2\rfloor}$. This bound is also the best possible, and a simple construction showing this can be obtained as follows. Say that a family $\mathcal{F}\subset 2^{[n]}$ is \emph{atomic}, if there exist disjoint sets $A_1,\dots,A_d\subset [n]$ such that $\mathcal{F}$ is the family of all sets $F$ satisfying that either $A_i\subset F$ or $A_i\cap F=\emptyset$ for every $i\in [d]$, and $F$ contains no element not covered by the sets $A_i$. 
	The sets $A_1,\dots,A_d$ are called the {\em atoms} of $\mathcal{F}$. 
	Also, let $S(n,\ell)$ be the atomic family for which $d=\lfloor n/\ell\rfloor$ and all $A_i, i\in [d]$ have size exactly $\ell$. Note that $|S(n,\ell)|=2^{\lfloor n/\ell\rfloor}$, and the size of the intersection of any number of sets in $S(n,\ell)$ is divisible by $\ell$. Therefore, the family $S(n,2)$ is an eventown of size $2^{\lfloor n/2\rfloor}$. 
	Moreover, any eventown family can be completed to a maximal one of size $2^{\lfloor n/2\rfloor}$, see e.g. the book of Babai and Frankl \cite{BF92}, which is also a general reference on intersection problems.
	
	In general, one might be tempted to conjecture that the maximal families $\mathcal{F}\subset 2^{[n]}$, whose all pairwise intersections are divisible by $\ell$, have size close to $2^{(1+o(1))n/\ell}$. However, this turns out to be far from the truth. Frankl and Odlyzko \cite{FO83} proved that if there exists a Hadamard matrix of order $4\ell$, then there exists such a family of size $2^{\Omega(n\log \ell/\ell)}$, and this bound is also the best possible up to the constant factor. On the other hand, it follows from a result of Deza, Erd\H{o}s and Frankl \cite{DEF78}, proved also by Frankl and Tokushige \cite{FT16},  that if we consider uniform families, that is, $\mathcal{F}\subset [n]^{(r)}$, then $|\mathcal{F}|\leq \binom{\lfloor n/\ell\rfloor}{ r/\ell}$ if $n$ is sufficiently large given $r$ and $\ell\mid r$. This bound is also the best possible as witnessed by the family $\mathcal{F}=[n]^{(r)}\cap S(n,\ell)$. Let us emphasize that the condition that $n$ must be large compared to $r$ is necessary, otherwise this would contradict the aforementioned construction of Frankl and Odlyzko.
	
	Despite all the above, if we require that the intersection of \emph{any number} of sets in $\mathcal{F}\subset 2^{[n]}$ must have size divisible by $\ell$, then it is not difficult to show that $|\mathcal{F}|\leq 2^{\lfloor n/\ell\rfloor}$ for any $n$ and $\ell$. Moreover, in this case, $\mathcal{F}$ is contained in some isomorphic copy of $S(n,\ell)$ (we say that two families in $2^{[n]}$ are isomorphic if they are equal up to a permutation of $[n]$). In 1983, Frankl and Odlyzko \cite{FO83} asked whether a similar conclusion holds if we only require that the intersection of any $k$ distinct sets in $\mathcal{F}$ has size divisible by $\ell$, where $k$ is some constant only depending on $\ell$. More precisely, they conjectured that for some $k$, we must have $|\mathcal{F}|\leq 2^{(1+o(1))n/\ell}$ for such a family $\mathcal{F}$. Until recently, it was not even known if the bound $2^{O(n\log \ell/\ell)}$ can be improved for any constant $k$. Indeed, while there are many tools to handle pairwise intersections as they correspond to the scalar product of characteristic vectors, $k$-wise intersections are usually harder to analyse, see, e.g.,  \cite{FS,GS02,MST,SV18,SV,V99} for related results. Also, it was shown in \cite{SV18} that if the conjecture is true, $k$ must depend on $\ell$. In particular, if $\ell$ is a power of $2$, there exist families $\mathcal{F}\subset 2^{[n]}$ such that the intersection of any $k$ sets in $\mathcal{F}$ has size divisible by $\ell$, and $|\mathcal{F}|\geq 2^{c_{k}n\log \ell/\ell}$, where $c_{k}>0$ is a constant only depending on $k$. In this paper we resolve the conjecture of Frankl and Odlyzko in the following strong form.
	
	\begin{theorem}\label{thm:main0}
		Let $\ell$ be a positive integer, then there exists $k=k(\ell)$ such that for every positive integer $n$ the following holds. Let $\mathcal{F}\subset 2^{[n]}$ such that the intersection of any $k$ distinct elements of $\mathcal{F}$ is divisible by $\ell$. Then $|\mathcal{F}|\leq 2^{\lfloor n/\ell\rfloor}+c$, where $c=c(\ell,k)$ is a constant, and $c=0$ if $\ell\mid n$ and $n$ is sufficiently large.
	\end{theorem}
	
	Note that the error term $c$ is needed if $\ell\nmid n$. Indeed, in this case $S(n,\ell)$ is not extremal, one can add a constant number of sets contained in the nonempty set not covered by members of $S(n,\ell)$ while retaining the property that the intersection of every $k$ distinct sets has size divisible by $\ell$.
	
	\subsection{Stability}
	
	As we mentioned above, maximum size eventowns are not unique. In particular, any eventown can be completed to an eventown of size $2^{\lfloor n/2\rfloor}$. However, as it was proved in \cite{SV18}, if we require that the intersection of any three sets is even sized, then $S(n,2)$ is the unique family achieving the maximum. Moreover, in this case we have stability, that is, if $\mathcal{F}\subset 2^{[n]}$ has this property and $|\mathcal{F}|\geq (1-\epsilon) 2^{\lfloor n/2\rfloor}$ for some small $\epsilon$, then $\mathcal{F}$ is a subfamily of some isomorphic copy of $S(n,2)$. 
	
	In Theorem \ref{thm:main0}, we also have stability. More precisely, if $\mathcal{F}\subset 2^{[n]}$ such that the intersection of  any $k$ distinct elements of $\mathcal{F}$ is divisible by $\ell$, and $|\mathcal{F}|> \frac{1}{2}\cdot 2^{\lfloor n/\ell\rfloor}$, then one can remove a constant number of sets  from $\mathcal{F}$ to make it a subfamily of some isomorphic copy of $S(n,\ell)$. However, somewhat surprisingly, a much more robust form of stability also holds. We show that under the substantially weaker condition $|\mathcal{F}|>2^{\alpha n}$, the family $\mathcal{F}$ already highly resembles a subfamily of $S(n,\ell)$ if $k$ is chosen appropriately with respect to $\alpha$ and $\ell$. The following result showing this seems to be new even for $\ell=2$.
	
	Say that a family $\mathcal{F}\subset 2^{[n]}$ is \emph{$k$-closed  $\mathrm{(mod }\ \ell)$} if the intersection of any $k$ (not necessarily distinct) sets in $\mathcal{F}$ has size divisible by $\ell$. Later, we show that if $\mathcal{F}$ has the property that any $k$ {\em distinct} sets in $\mathcal{F}$ have an intersection of size divisible by $\ell$, then $\mathcal{F}$ can be made $k$-closed by removing $O(n)$ elements. The next statement will be more convenient to state for $k$-closed families, however, this only makes a small difference by the previous claim.
	If $X\subset [n]$ then $\mathcal{F}|_{X}=\{F\cap X: F \in \mathcal{F}\}$ denotes the \emph{projection} of $\mathcal{F}$ onto $X$.
	
	\begin{theorem}\label{thm:stab}
		Let $\epsilon>0$, and let $\ell$ be a positive integer, then there exists $k=k(\ell,\epsilon)$ such that the following holds. Let $\mathcal{F}\subset 2^{[n]}$ be $k$-closed $\mathrm{(mod }\ \ell)$. Then there exist $X\subset [n]$ such that $|\mathcal{F}|_{X}|\geq2^{-\epsilon n}|\mathcal{F}|$, and  $\mathcal{F}|_{X}$ is a subfamily of an isomorphic copy of $S(|X|,\ell)$.
	\end{theorem}
	
	The following construction shows that this theorem is also optimal in a certain sense. Suppose that $\ell=p$ is a prime and $2p^{k+1}$ divides $n$. Partition $\{1,\dots,n/2\}$ into sets $A_1,\dots,A_q$ of size $p$, and partition $\{n/2+1,\dots,n\}$ into sets  $B_1,\dots,B_r$ of size $p^{k+1}$. Let $\mathcal{F}$ be the family of sets $F$ of the following form.
	\begin{itemize}
		\item For $i\in [q]$, either $A_i\subset F$ or $A_i\cap F=\emptyset$.
		\item For $j\in[r]$, identify $B_j$ with the vector space $\mathbb{F}_p^{k+1}$. Then $F\cap B_j$ is a $k$-dimensional subspace of $\mathbb{F}_p^{k+1}$.
	\end{itemize}
	Clearly, we have $|\mathcal{F}|=2^{q}(\frac{p^{k+1}-1}{p-1})^{r}>2^{n/2p}p^{nk/2p^{k+1}}$. Also, if $F_1,\dots,F_k\in\mathcal{F}$, then $|F_1\cap\dots\cap F_k|$ is divisible by $p$. Indeed, $|F_1\cap\dots\cap F_k\cap A_i|\in \{0,p\}$, and $F_1\cap\dots\cap F_k\cap B_j$ is a subspace of $\mathbb{F}_p^{k+1}$ of dimension at least $1$, so its size is also divisible by $p$. Finally, if $X\subset [n]$ is such that $\mathcal{F}|_{X}$ is a subfamily of some isomorphic copy of $S(|X|,p)$, then $X\subset \{1,\dots,n/2\}$ and $|\mathcal{F}|_{X}|\leq 2^{n/2p}<2^{-n(\log_2 p)k/2p^{k+1}} |\mathcal{F}|.$
	
	\subsection{Tools}
	
	The main technical tool we develop to prove Theorem \ref{thm:main0} is a structure theorem for set systems with small `doubling'. This result is similar in spirit to the famous Freiman-Ruzsa type theorems in additive combinatorics. These theorems describe the approximate structure of the subsets $A$ of a group with the property that the sum-set $A + A$ has size not much larger than $A$ (i.e. $A$ has small doubling). The theorem of Freiman and Ruzsa \cite{Freiman,Ruzsa} states that a subset of the integers with small doubling must be contained in a so-called generalized arithmetic progression of bounded rank (see also \cite{TV}). This classical result was subsequently generalized to abelian groups by Green and Ruzsa \cite{GR07} and to all groups by Breuillard, Green and Tao \cite{BGT12} (see also the survey \cite{BGT_Survey}). These results show that sets with small doubling must be close in structure to one of few natural examples. 
	
	Given a set-family $\mathcal{F}$, our measure for the size of $\mathcal{F}$ will be the dimension of the subspace $\langle \mathcal{F} \rangle$ spanned by the characteristic vectors of the sets in $\mathcal{F}$ over some field $\mathbb{F}$. 
	Let $\mathcal{F}\cdot\mathcal{F}=\{A\cap B:A,B\in\mathcal{F}\}$. Note that, by definition, $\mathcal{F} \subset \mathcal{F} \cdot \mathcal{F}$. 
	What can we say about $\mathcal{F}$ if the dimension of $\langle \mathcal{F} \cdot \mathcal{F} \rangle$ is not much larger than that of $\langle \mathcal{F} \rangle$? Observe that if $S$ is an atomic set-family, then $S \cdot S = S$.
	Our structure theorem shows that this is essentially the only possible example: any set-family $\mathcal{F}$ with small `doubling' must be close to being atomic. To make this more precise, we need the following definition. Given $i,j\in [n]$, say that $i$ and $j$ are {\em twins} for $\mathcal{F}$ if every $F \in \mathcal{F}$ either contains both $i,j$ or none of them, and there is at least one $F \in \mathcal{F}$ such that $i,j \in F$. Note that being twins is an equivalence relation (on the set of $i \in [n]$ which are contained in at least one $F \in \mathcal{F}$). A set of coordinates $T\subset [n]$ is called a \emph{set of twins} if any pair of elements in $T$ are twins, or $|T|=1$. Also, say that $T$ is a maximal set of twins if it is a complete equivalence class of the twins relation. We can now state our structure theorem.

	\begin{theorem}\label{cor:stability}
		Let $\mathcal{F}\subset\{0,1\}^{n}$, let $\mathbb{F}$ be a field, and suppose that $\dim\langle \mathcal{F} \rangle = d$ and $\dim \langle \mathcal{F}\cdot \mathcal{F} \rangle=d+h$. Then $[n]$ can be partitioned into $d+1$ sets $A_1,\dots,A_d,B$ such that $A_{i}$ is a maximal set of twins for $\mathcal{F}$ for $i\in [d]$, and $\dim \langle \mathcal{F}|_{B} \rangle \leq 2h$.
	\end{theorem}
	
	\section{Small `doubling' and twins}
	
	In this section we establish Theorem \ref{cor:stability}. We will actually prove a more general statement about arbitrary vector spaces. Let us introduce some notation. 
	
	As usual, $\mathbb{Z}_{\ell}$ denotes the ring of integers modulo $\ell$, and if $p$ is a prime, we write $\mathbb{F}_p$ instead of $\mathbb{Z}_p$ to emphasize that it is also a field. Fix any commutative ring $\mathcal{R}$ with a unity  (in our case, $\mathcal{R}$ will be either a field or  $\mathbb{Z}_\ell$ for some positive integer $\ell$). For a vector $v \in \mathcal{R}^n$, we use $v(i)$ to denote the $i$th coordinate of $v$. The {\em support} of $v$ is $\{i \in [n] : v(i) \neq 0\}$. For $\mathcal{F}\subset \mathcal{R}^{n}$, we use $\langle \mathcal{F}\rangle$ to denote the \emph{span} (i.e. the set of all linear combinations) of the elements of $\mathcal{F}$. If $\mathcal{R}=\mathbb{Z}_{\ell}$, we might write $\langle \mathcal{F}\rangle_{\ell}$ instead of $\langle \mathcal{F}\rangle$ if $\mathcal{R}$ is not clear from the context. If $A\subset [n]$ and $v \in \mathcal{R}^n$, then $v|_{A}\in \mathcal{R}^{A}$ denotes the restriction of $v$ to the coordinates in $A$, and if $\mathcal{F}\subset \mathcal{R}^{n}$, then $\mathcal{F}|_{A}=\{v|_{A}: v \in \mathcal{F}\}$.
	
	Given vectors $v,w \in \mathcal{R}^n$, let $v\cdot w$ be the vector in $\mathcal{R}^n$ defined as $(v\cdot w)(i)=v(i)w(i)$ for $i\in [n]$. Note that if $v$ and $w$ are characteristic vectors of sets $A$ and $B$, then $v\cdot w$ is the characteristic vector of $A\cap B$. 
	For $V,W \subset \mathcal{R}^{n}$, let $V \cdot W = \{v \cdot w : v \in V, w \in W\}$. 
	Given $V \subset \mathcal{R}^{n}$ and $i,j\in [n]$, say that $i$ and $j$ are \emph{twins} for $V$ if  $v(i)=v(j)$ for all $v\in V$ and $v(i)\neq 0$ for at least one $v\in V$. Observe that if $V$ is a subspace generated by some family $\mathcal{F} \subset \{0,1\}^n$, then this definition of twins exactly coincides with the one given in the previous section.

	\begin{theorem}\label{lemma:stability}
		Let $\mathbb{F}$ be a field, $V<\mathbb{F}^{n}$, $d=\dim(V)$ and 
		$\dim(\langle V\cup (V\cdot V)\rangle)=d+h$. Then $[n]$ can be partitioned into $d+1$ sets $A_1,\dots,A_{d},B$ such that $A_{i}$ is a maximal set of twins for $V$ for each $i\in [d]$, and $\dim(V|_{B})\leq 2h$. 
	\end{theorem} 
	
	\begin{proof}
		For $i,j\in [n]$, say that $i$ and $j$ are \emph{siblings} for $V$ if there exists $\lambda\in\mathbb{F},\lambda\neq0$ such that $v(i)=\lambda v(j)$ for all $v\in V$, and $v(i)\neq 0$ for at least one $v\in V$. 
		Equivalently, $i$ and $j$ are siblings if for every basis $V_0$ of $V$ there exists $\lambda\in\mathbb{F},\lambda\neq0$ such that $v(i)=\lambda v(j)$ for all $v\in V_0$, and $v(i)\neq 0$ for at least one $v\in V_0$ (in which case $\lambda$ will be the same for all bases of $V$).  Note that if $\lambda = 1$, then $i$ and $j$ are twins.
		
		Let $v_1,\dots,v_d\in V$ be a basis of $V$, and let $M$ be the $d\times n$ matrix, whose rows are $v_1,\dots,v_d$. It is possible to choose the basis $v_1,\dots,v_d$ such that after possibly rearranging the columns of $M$, the restriction of $M$ to the first $d$ columns is a diagonal matrix. 
		
		We can assume without loss of generality that $M$ has no all-zero columns, i.e. that there is no coordinate $i \in [n]$ such that $v(i) = 0$ for all $v \in V$. Indeed, removing such an index does not change the dimension of $V$ and $\langle V\cup (V\cdot V)\rangle$, and we may include it in $B$ without increasing the dimension of $V|_{B}$.
		
		Note that the indices $i\in [d]$ and $j\in [n]\setminus [d]$ are siblings for $V$ if and only if $v_i(j)\neq 0$ and the $j$-th column of $M$ contains exactly one nonzero entry (namely, the entry $v_i(j)$). For $i\in [d]$, let $S_i$ contain all the siblings of $i$, and let $B'=[n]\setminus (S_1\cup\dots\cup S_d)$. Note that for every $j\in B'$, the $j$-th column of $M$ contains at least two nonzero entries.
		
		Let $r=\dim(V|_{B'})$. Choose some subset $C\subset B'$ such that the columns of $C$ form a basis of the vector space spanned by the columns of $B'$. Then $|C|=r$ and $\dim(V|_{C})=r$ (because the row space has the same dimension as the column space). See Figure \ref{fig1} for an illustration of the matrix $M$ and sets $S_1,\dots,S_d,B',C$. Furthermore, let $w_{i}=v_{i}|_{C}$ for $i\in [d]$, and observe that $w_1,\dots,w_d$ spans $\mathbb{F}^{C}$ (because $|C| = r$ and $\dim(V|_{C})=r$). For $c\in C$, let $\mathds{1}_c\in\mathbb{F}^{C}$ be the characteristic vector of the single element set $\{c\}$.
		
		\begin{claim}\label{claim:c'}
			For every $c\in C$ there exists $\alpha\in \mathbb{F}$, $c'\in C, c'\neq c$, and two vectors of coefficients $x,y\in \mathbb{F}^{d}$ such that $x$ and $y$ have disjoint supports and 
			$$\mathds{1}_{c}+\alpha\mathds{1}_{c'}=\left(\sum_{i=1}^d x(i)w_i\right)\cdot \left(\sum_{i=1}^d y(i)w_i\right).$$
		\end{claim}	
		\noindent
		Here, as always, $\cdot$ denotes the coordinate-wise product of vectors.
		
		\begin{figure}
			\centering
			\begin{tikzpicture}
				\node at (-3.5,0.7) {$v_1$};
				\node at (-3.5,0.233) {$v_2$};
				\node at (-3.5,-0.233) {$v_3$};
				\node at (-3.5,-0.7) {$v_4$};
				\matrix at (0,0) [matrix of math nodes,left delimiter=(,right delimiter=),row sep=0cm,column sep=0cm] (m) {
					1 &   &   &   &   & 2 &   &   &   & 2 &   & 1\\
					& 1 &   &   & 2 & 1 &   & 2 & 1 &   & 2 & 2\\
					&   & 1 &   &   &   &   & 1 & 1 &   &   & 1\\
					&   &   & 1 &   &   & 1 &   & 1 &   &   & 2\\};
				\draw[dashed] (-0.9,-1) -- (-0.9,1);
				\matrix (m) at (8,0) [matrix of math nodes,left delimiter=(,right delimiter=),row sep=0cm,column sep=0cm] (m) {
					1 & 2 &   &   &   &   &   &   & 2 &   &   & 1\\
					&   & 1 & 2 & 2 &   &   &   & 1 & 2 & 1 & 2\\
					&   &   &   &   & 1 &   &   &   & 1 & 1 & 1\\
					&   &   &   &   &   & 1 & 1 &   &   & 1 & 2\\};
				\draw[dashed] (8.9,-1) -- (8.9,1);
				\draw[dashed] (10.25,-1) -- (10.25,1);
				\draw [thick,decorate,decoration={calligraphic brace,amplitude=4pt}] (5.4,1) -- (6.2,1);
				\node at (5.8,1.5) {$S_1$};
				\draw [thick,decorate,decoration={calligraphic brace,amplitude=4pt}] (6.4,1) -- (7.45,1);
				\node at (6.925,1.5) {$S_2$};
				\node at (7.8,1.5) {$S_3$};
				\draw [thick,decorate,decoration={calligraphic brace,amplitude=4pt}] (8.1,1) -- (8.8,1);
				\node at (8.45,1.5) {$S_4$};
				\draw [thick,decorate,decoration={calligraphic brace,amplitude=4pt}] (8.9,1) -- (10.25,1);
				\node at (9.425,1.5) {$C$};
				\draw [thick,decorate,decoration={calligraphic brace,amplitude=4pt}]  (10.6,-1) -- (8.9,-1);
				\node at (9.6,-1.5) {$B'$};
			\end{tikzpicture}
			\caption{On the left, a matrix $M$ over $\mathbb{F}_3$, where blank entries denote 0. On the right, we rearranged the columns of $M$ for better visualization.}
			\label{fig1}
		\end{figure}
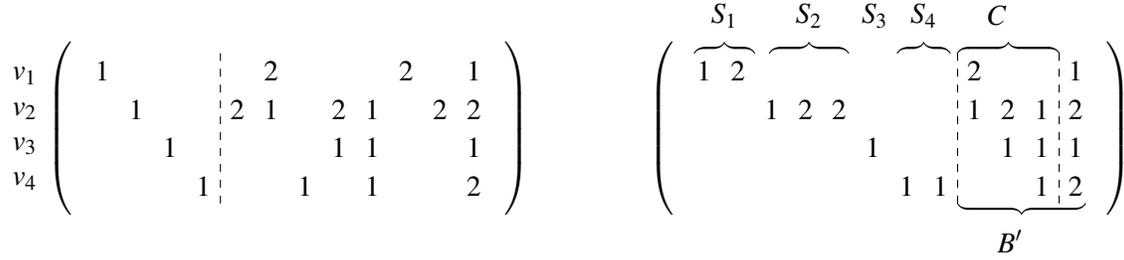

		\begin{proof}
			Let $K\subset [d]$ be the set of indices $i \in [d]$ such that $w_{i}(c)\neq 0$. Recall that $|K|\geq 2$, as every column indexed by an element of $C\subset B'$ contains at least two nonzero entries. As $w_1,\dots,w_d$ span $\mathbb{F}^{C}$, we can choose an $r$-element subset of $\{w_1,\dots,w_d\}$ forming a basis of $\mathbb{F}^{C}$. Without loss of generality, suppose that $w_1,\dots,w_r$ is a basis. Then, for every $u\in C$, we can write $\mathds{1}_u=\sum_{i=1}^r\lambda_{u,i} w_i$ with suitable $\lambda_{u,1},\dots,\lambda_{u,r}\in\mathbb{F}$. Consider three cases.
			
			\begin{description}
				\item[Case 1.] There exists $k\in K$ such that $k\not\in [r]$. 
				
				Recall that, by definition of $K$, $w_k(c)\not =0$. Set $\alpha=0$, $x=(\lambda_{c,1},\dots,\lambda_{c,r},0,\dots,0) \in \mathbb{F}^d$, and define $y \in \mathbb{F}^d$ as 
				$$y(i)=\begin{cases}\frac{1}{w_k(c)} &\mbox{if }i=k,\\
					0                &\mbox{otherwise}.
				\end{cases}$$
				Then $\sum_{i=1}^d x(i)w_i=\mathds{1}_{c}$ and $\sum_{i=1}^d y(i)w_i=\frac{1}{w_k(c)}w_k$, which implies that $(\sum_{i=1}^d x(i)w_i)\cdot (\sum_{i=1}^d y(i)w_i)=\mathds{1}_c$. Furthermore, the support of $x$ is contained in $[r]$, while the support of $y$ is $\{k\}$, so $x$ and $y$ have disjoint supports. We see that $\alpha,x,y,c'$ satisfy the requirements for any choice of $c'\in C$, $c'\neq c$. 
				
				\item[Case 2.] $K\subset [r]$ and there exists $k\in K$ such that $\lambda_{c,k}=0$. 
				
				Define $\alpha,x,y$ the same way as in the previous case. Note that the only difference is that the support of $x$ is now contained in $[r]\setminus\{k\}$, so $x$ and $y$ still have disjoint supports. Therefore, $\alpha,x,y,c'$ suffice for any $c'\in C$, $c'\neq c$.
				
				\item[Case 3.] $K\subset [r]$ and $\lambda_{c,k}\neq 0$ for every $k\in K$.
				
				We claim that there exists some $c'\in C,c'\neq c$ such that not all $|K|$ coefficients $\lambda_{c',k}$ for $k\in K$ vanish. 
				Indeed, suppose by contradiction that $\lambda_{c',k} = 0$ for all $k\in K$ and $c' \in C \setminus \{c\}$. By the definition of the coefficients $\lambda_{c',i}$, this means that $\mathds{1}_{c'}$ belongs to the span of $(w_j)_{j \in [r] \setminus K}$ for every $c' \in C \setminus \{c\}$. But the $r-1$ vectors $(\mathds{1}_{c'})_{c'\in C \setminus \{c\}}$ are linearly independent, while $r - |K| \leq r-2$, so this is impossible. 
				
				Choose $c' \in C \setminus \{c\}$ and $k\in K$ such that $\lambda_{c',k}\neq 0$. Let $\beta=-\lambda_{c,k}/\lambda_{c',k}$, 
				$$x=(\lambda_{c,1}+\beta\cdot \lambda_{c',1},\dots,\lambda_{c,r}+\beta\cdot\lambda_{c',r},0,\dots,0),$$ and define $y$ as
				$$y(i)=\begin{cases}\frac{1}{w_k(c)} &\mbox{if }i=k,\\
					0                &\mbox{otherwise}
				\end{cases}.$$
				Then $\sum_{i=1}^d x(i)w_i=\mathds{1}_c+\beta \mathds{1}_{c'}$ and $\sum_{i=1}^d y(i)w_i=\frac{1}{w_k(c)}w_k$. Therefore, we have 
				$$\left(\sum_{i=1}^d x(i)w_i\right)\cdot \left(\sum_{i=1}^d y(i)w_i\right)=\mathds{1}_c+\frac{\beta w_k(c')}{w_k(c)}\mathds{1}_{c'}.$$ Also, by the choice of $\beta$, we have $x(k)=0$, so the support of $x$ is contained in $[r]\setminus\{k\}$, while the support of $y$ is $\{k\}$. Therefore, $c',\alpha=\frac{\beta w_k(c')}{w_k(c)},x,y$ satisfy the requirements of the claim.
			\end{description}
			This completes the proof of Claim \ref{claim:c'}. See Figure \ref{fig2} for an illustration of $M|_{C}$ and the three cases.
		\end{proof}
		
		\begin{figure}
			\centering
			\begin{tikzpicture}
				\matrix at (0,0) [matrix of math nodes,left delimiter=(,right delimiter=),row sep=0cm,column sep=0cm] {
					& 1 & 1 \\
					1 & 1 & 2 \\
					2 &   & 1 \\
					2 &   &   \\
					1 &   &   \\
					2 &   &   \\};
				\draw[dashed] (-1,0) -- (1,0);
				\node at (-1.5,1.15) {$w_1$};
				\node at (-1.5,0.2) {$w_r$};
				\node at (-1.5,-0.2) {$w_{r+1}$};
				\node at (-1.5,-1.15) {$w_{d}$};
				\node at (-0.45,1.8) {$c_1$};
				\node at (0,1.8) {$c_2$};
				\node at (0.45,1.8) {$c_3$};
				\draw [thick,decorate,decoration={calligraphic brace,amplitude=4pt}] (0.9,-1.7) -- (-0.9,-1.7);
				\node at (0,-2) {$C$};
				
				\node[anchor=west] at (3,-0.6) {\large $\mathds{1}_{c_3}=w_1+2w_2+2w_3$};
				\node[anchor=west] at (3,0.6) {\large $\mathds{1}_{c_1}=w_1+2w_2+w_3$};
				\node[anchor=west] at (3,0) {\large $\mathds{1}_{c_2}=\ \ \ \ \ \ \ \ \ \ w_2+w_3$};
			\end{tikzpicture}
			\caption{An example of the matrix $M|_{C}$ in the field $\mathbb{F}_3$, where blank entries denote 0. The column $c=c_1$ falls into Case 1., $c_2$ falls into Case 2., $c_3$ falls into Case 3.}
			\label{fig2}
		\end{figure}
		
		For $c\in C$, let $x,y\in\mathbb{F}^{d}$ be the vectors of coordinates  satisfying the conditions of the previous claim. Also, let $X=\sum_{i=1}^d x(i)v_i$, $Y=\sum_{i=1}^d y(i)v_i$, and set $z_{c}=X\cdot Y$. Then $z_{c}\in V\cdot V$ and $z_{c}|_{C}=\mathds{1}_c+\alpha\mathds{1}_{c'}$ for some $c'\in C,c'\neq c$ and $\alpha\in \mathbb{F}$, as guaranteed by Claim \ref{claim:c'}. Observe that the support of $X$ on $[n]\setminus B'$ is $\bigcup_{i\in [d], x(i)\neq 0}S_i$, while the support of $Y$ on $[n]\setminus B'$ is $\bigcup_{i\in [d], y(i)\neq 0}S_i$ (indeed, recall that $[n]\setminus B'$ is the union of $S_1,\dots,S_d$, and every column in $S_i$ has a non-zero entry in the $i$th row and zeros in all other rows).
		Hence, as $x$ and $y$ have disjoint supports, $X$ and $Y$ have disjoint supports on $[n]\setminus B'$. This means that $z_c=X\cdot Y$ vanishes on $[n]\setminus B'$.
		
		 Let $W$ be the vector-space generated by the vectors $(z_c)_{c\in C}$, and let us record that $W$ vanishes on $[n]\setminus B'$.
		\begin{claim}
			$\dim(W)\geq r/2.$
		\end{claim}
		\begin{proof}
			Suppose that $\dim(W) < r/2$, and without loss of generality, let $D\subset C$ such that $\{z_c\}_{c\in D}$ is a basis of $W$ and $|D|<r/2 = |C|/2$. 
			Recall that for each $c \in C$, $z_{c}|_C = \mathds{1}_c+\alpha\mathds{1}_{c'}$ for some $c'\in C \setminus \{c\}$ and $\alpha\in \mathbb{F}$.
			As $z_{c}|_C$ vanishes on all but at most two coordinates, there exists $u\in C$ such that $z_c$ vanishes on $u$ for every $c\in D$. But $z_u$ does not vanish on $u$, so $z_u$ is not in the span of $\{z_c\}_{c\in D}$, contradicting that $\{z_c\}_{c\in D}$ is a basis for $W$.
		\end{proof}
		
		Let $I$ be the set of indices $i \in [d]$ such that the entries in $v_i|_{S_i}$ are {\em not} all equal, that is, $S_i$ is not a set of twins. For $i\in [d]$, let $A_{i}$ be a maximal set of twins in $S_{i}$, then $S_{i}\neq A_{i}$ if and only if $i\in I$. If $i\in I$, and $v_i|_{A_i}$ is the constant $s$ vector, then define $v_{i}'=sv_{i}-v_{i}\cdot v_{i}\in \langle V\cup (V\cdot V)\rangle$. Observe that $v_{i}'$ vanishes on $A_{i}$ and on $S_{j}$ for every $j\in [d]\setminus\{i\}$, and $v_{i}'$ does not vanish on $S_{i}\setminus A_{i}$. 
		Crucially, this implies that the $d+|I|$ vectors $v_1,\dots,v_{d},\{v_{i}'\}_{i\in I}$ are not only linearly independent, but their restrictions on $[n]\setminus B'$ are also linearly independent. 
		This means that setting $V'=\langle \{v_{i}:i\in [d]\}\cup\{v_{i}':i\in I\}\rangle$, we have $\dim(V')=d+|I|$, and the only element of $V'$ vanishing on $[n]\setminus B'$ is $\mathbf{0}$. On the other hand, every element of $W$ vanishes on $[n]\setminus B'$, which gives $W\cap V'=\{\mathbf{0}\}$.  See Figure \ref{fig3} for an illustration of the vectors $v_1,\dots,v_d,\{v'_i\}_{i\in I},\{z_c\}_{c\in C}$.
		
		But then as $V'+W < \langle V\cup (V\cdot V)\rangle$, we have
		$$\dim(\langle V\cup (V\cdot V)\rangle)\geq\dim (V'+W)=\dim(V')+\dim(W)\geq d+|I|+r/2.$$
		On the other hand, $\dim(\langle V\cup (V\cdot V)\rangle) = d+h$ by assumption.
		Therefore, $|I|+r/2\leq h$. Let $B=B'\cup\bigcup_{i\in I}(S_{i}\setminus A_{i})$, then $\dim(V|_{B})\leq r+|I|\leq 2h$, so the sets $A_{1},\dots,A_{d},B$ satisfy the desired properties.
	\end{proof}
	
	\begin{figure}
		\centering
		\begin{tikzpicture}
			\node at (4.6,1.8) {$v_1$};
			\node at (4.6,1.333) {$v_2$};
			\node at (4.6,0.866) {$v_3$};
			\node at (4.6,0.4) {$v_4$};
			\node at (4.6,-0.066) {$z_{c_1}$};
			\node at (4.6,-0.533) {$z_{c_2}$};
			\node at (4.6,-1) {$z_{c_3}$};
			\node at (4.6,-1.466) {$v_{1}'$};
			\node at (4.6,-1.933) {$v_2'$};
			\draw[rounded corners,color=white,fill=white!80!red] (5.3,2.3) rectangle (5.75,-2.3);
			\draw[rounded corners,color=white,fill=white!80!blue] (6.7,2.3) rectangle (7.5,-2.3);
			\draw[rounded corners,color=white,fill=white!80!green] (7.55,2.3) rectangle (8,-2.3);
			\draw[rounded corners,color=white,fill=white!80!brown] (8.02,2.3) rectangle (8.82,-2.3);
			\node at (5.525,-2.9) {$A_1$};
			\node at (7.18,-2.9) {$A_2$};
			\node at (7.81,-2.9) {$A_3$};
			\node at (8.42,-2.9) {$A_4$};
			\draw [thick,decorate,decoration={calligraphic brace,amplitude=4pt}] (7.5,-2.4) -- (6.7,-2.4);
			\draw [thick,decorate,decoration={calligraphic brace,amplitude=4pt}] (8.82,-2.4) -- (8.02,-2.4);
			
			\matrix at (8,0) [matrix of math nodes,left delimiter=(,right delimiter=),row sep=0cm,column sep=0cm] {
				1 & 2 &   &   &   &   &   &   & 2 &   &   & 1\\
				&   & 1 & 2 & 2 &   &   &   & 1 & 2 & 1 & 2\\
				&   &   &   &   & 1 &   &   &   & 1 & 1 & 1\\
				&   &   &   &   &   & 1 & 1 &   &   & 1 & 2\\
				&   &   &   &   &   &   &   & 1 &   &   & 1\\
				&   &   &   &   &   &   &   &   & 1 & 2 & 1\\
				&   &   &   &   &   &   &   &   & 2 & 1 & 2\\	  
				& 1 &   &   &   &   &   &   & 1 &   &   &  \\
				&   &  1 &   &  &   &   &   & 1 &   & 1 &  \\	  
			};
			
			\draw[dashed] (8.9,-2.3) -- (8.9,2.3);
			\draw[dashed] (10.25,-2.3) -- (10.25,2.3);
			\draw[dashed] (5.2,0.23) -- (10.8,0.23);
			\draw[dashed] (5.2,-1.3) -- (10.8,-1.3);
			\draw [thick,decorate,decoration={calligraphic brace,amplitude=4pt}] (5.3,2.4) -- (6.2,2.4);
			\node at (5.75,2.9) {$S_1$};
			\draw [thick,decorate,decoration={calligraphic brace,amplitude=4pt}] (6.4,2.4) -- (7.45,2.4);
			\node at (6.925,2.9) {$S_2$};
			\node at (7.81,2.9) {$S_3$};
			\draw [thick,decorate,decoration={calligraphic brace,amplitude=4pt}] (8.02,2.4) -- (8.82,2.4);
			\node at (8.42,2.9) {$S_4$};
			\node at (9.1,2.9) {$c_1$};
			\node at (9.525,2.9) {$c_2$};
			\node at (9.95,2.9) {$c_3$};
			\draw [thick,decorate,decoration={calligraphic brace,amplitude=4pt}]  (10.6,-2.4) -- (8.9,-2.4);
			\node at (9.75,-2.9) {$B'$};
			\draw [thick,decorate,decoration={calligraphic brace,amplitude=4pt}] (11.3,2.3) -- (11.3,0.2);
			\node at (12.8,1.25) {$M$};
		\end{tikzpicture}
		\caption{An example of a matrix $M$ over $\mathbb{F}_3$, together with the vectors $z_{c}$ for $c\in C=\{c_1,c_2,c_3\}$, and the vectors $v_{i}'$ for $i\in I$.}
		\label{fig3}
	\end{figure}
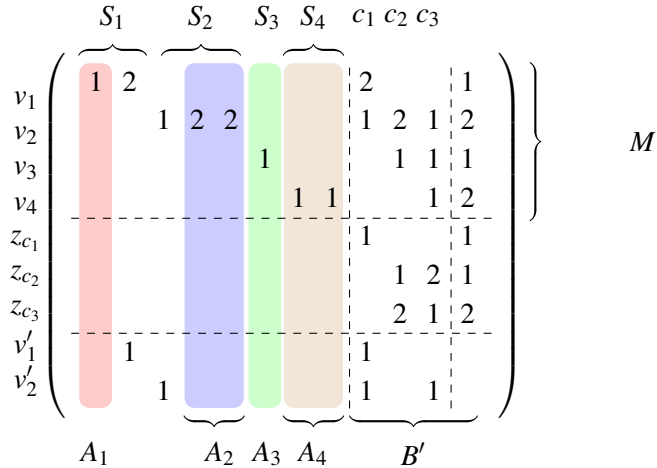

	\section{\texorpdfstring{$k$}{k}-closed families are atomic}
	
	In this section, we prove a theorem which implies Theorem \ref{thm:main0} after a small amount of work. Before stating it, we need some additional notation.  Say that $\mathcal{F}$ is \emph{non-reducible} if $\mathcal{F}$ does not vanish on any of the coordinates (namely, if there is no $i$ such that $v(i) = 0$ for all $v \in \mathcal{F}$). Recall that $\mathcal{F} \cdot \mathcal{F}$ is the set of all $v \cdot w$, $v,w \in \mathcal{F}$, where $v \cdot w$ is the coordinate-wise product. We also put $v^k = v \cdot \dotsc\cdot v$ and $\mathcal{F}^{k}=\mathcal{F}\cdot \dotsc \cdot \mathcal{F}$, where the products contain $k$ terms. Finally, let $||v||=\sum_{i=1}^{n}v(i)$. 
	
	We say that a set $\mathcal{F}\subset \mathbb{Z}^{n}_{\ell}$ is \emph{$k$-closed} if $||v||=0$ for every $1\leq i\leq k$ and $v\in \mathcal{F}^{i}$. Note that if $\mathcal{F} \subset \{0,1\}^n$, then this is the same as saying that the intersection of any $k$ not necessarily distinct sets from $\mathcal{F}$ is divisible by $\ell$.  Let us collect some simple properties of the coordinate-wise product and $k$-closedness.
	
	\begin{claim}\label{claim:k-closed}
		\begin{enumerate}
			
			\item  If $\mathcal{F},\mathcal{F}'\subset \mathbb{Z}_{\ell}^{n}$, then $\langle \mathcal{F}\rangle\cdot \langle \mathcal{F}'\rangle\subset \langle \mathcal{F}\cdot\mathcal{F}'\rangle$.
			
			\item 	If $\mathcal{F}\subset \mathbb{Z}_{\ell}^{n}$ is $k$-closed, then $\langle \mathcal{F}\rangle$ is also $k$-closed.
			
			\item If $\mathcal{F}$ is $k$-closed, then $\mathcal{F}\cdot\mathcal{F}$ is $\lfloor k/2\rfloor$-closed.
		\end{enumerate}
	\end{claim}
	
	\begin{proof}
		1. Let $v\in\langle\mathcal{F}\rangle$ and $w\in\langle\mathcal{F}'\rangle$. Then there exists $a_x\in\mathbb{Z}_{\ell}$ for every $x\in\mathcal{F}$ such that $v=\sum_{x\in\mathcal{F}}a_x x$, and similarly, there exists $b_y\in\mathbb{Z}_{\ell}$ for every $y\in\mathcal{F}'$  such that $w=\sum_{y\in\mathcal{F}'}b_y y$. But then
		$$v\cdot w=\left(\sum_{x\in \mathcal{F}}a_x x\right)\cdot \left(\sum_{y\in\mathcal{F}'}b_y y\right)=\sum_{x\in\mathcal{F}}\sum_{y\in\mathcal{F}'}a_xb_y(x\cdot y)\in \langle \mathcal{F}\cdot\mathcal{F}'\rangle.$$
		
		2. Let $i\in [k]$. By repeatedly applying 1., we can write $\langle\mathcal{F}\rangle^{i}\subset \langle \mathcal{F}^{i}\rangle$.
		As $\mathcal{F}$ is $k$-closed, we have $||v||=0$ for every $v\in\mathcal{F}^{i}$. 
		As $|| \cdot ||$ is a linear function,
		we have $||w||=0$ for every $w\in \langle \mathcal{F}^{i}\rangle$. Hence, $||w||=0$ for every $1\leq i\leq k$ and $w\in \langle \mathcal{F}\rangle^{i}$, meaning that $\langle \mathcal{F}\rangle$ is $k$-closed.
		
		3. This follows trivially from the definition.
	\end{proof}
	
	\noindent
	Furthermore, let us record some simple properties of the twin relation. 
	\begin{claim}\label{claim:twins}
		Let $\mathcal{F}\subset \{0,1\}^{n}$.
		\begin{enumerate}
			\item If $\mathcal{F}$ is non-reducible, then the maximal sets of twins for $\mathcal{F}$ form a partition of $[n]$.
			
			\item For every $k\geq 1$, the family $\bigcup_{i=1}^{k} \mathcal{F}^{i}$ has the same pairs of twins as $\mathcal{F}$.
			
		\end{enumerate}
	\end{claim}
	\begin{proof}
		1. This is a consequence of the fact that the twin relation is an equivalence relation.
		
		2. Let $a,b\in [n]$, $a\neq b$. If $a$ and $b$ are twins in $\mathcal{F}$, then $v(a)=v(b)$ for every $v\in\mathcal{F}$, which implies that $v_1(a)\dots v_i(a)=v_1(b)\dots v_i(b)$ for every $i\in [k]$ and $v_1,\dots,v_i\in\mathcal{F}$. Hence, $a$ and $b$ are twins in $\mathcal{F}^{i}$ as well, and so in $\bigcup_{i=1}^{k}\mathcal{F}^i$. Also, if $a$ and $b$ are not twins, then there exists some $v\in\mathcal{F}\subset \bigcup_{i=1}^{k}\mathcal{F}^i$ such that $v(a)\neq v(b)$, so $a$ and $b$ are not twins in $\bigcup_{i=1}^{k}\mathcal{F}^i$.
	\end{proof}
	Recall that $S(n,\ell) \subset 2^{[n]}$ is the atomic set-family with $\lfloor n/\ell \rfloor$ atoms of size $\ell$ each. 
	The main result of this section is the following variant of Theorem \ref{thm:main0}. We show that if $\mathcal{F}\subset 2^{[n]}$ is such that the intersection of any $k$ not necessarily distinct elements of $\mathcal{F}$ has size divisible by $\ell$, then $|\mathcal{F}|\leq 2^{\lfloor n/\ell\rfloor}$, given $k$ is sufficiently large with respect to $\ell$. We also show that if $\mathcal{F}$ is close to being extremal, then $\mathcal{F}$ must be a subfamily of (an isomorphic copy of) $S(n,\ell)$. 
	\begin{theorem}\label{thm:main}
		Let $\ell$ be a positive integer, then there exists $k$ such that the following holds. Let $\mathcal{F}\subset \{0,1\}^{n}$ such that $\mathcal{F}$ is $k$-closed over $\mathbb{Z}_{\ell}$. Then $|\mathcal{F}|\leq 2^{\lfloor n/\ell\rfloor}$. Also, if $|\mathcal{F}|> 2^{\lfloor n/\ell\rfloor-1}$, then $[n]$ can be partitioned into sets $A_{1},\dots,A_{d},A'$ such that $A_{i}$ is a maximal set of twins for $\mathcal{F}$ for $i\in [d]$, $|A_{i}|=\ell$, $|A'|\leq \ell-1$, and $\mathcal{F}$ vanishes on $A'$. 
	\end{theorem}
	We start preparing the proof of Theorem \ref{thm:main} with a sequence of lemmas. 
	The Eventown theorem mentioned in the introduction can be easily extended to vector spaces, where we replace intersections with scalar products. We will make use of the following simple extension of this in which we consider certain bilinear forms instead of scalar product.
	
	\begin{lemma}\label{lemma:bilinear}
		Let $\mathbb{F}$ be a field, let $b_1,\dots,b_d\in \mathbb{F}$, let $z$ be the number of zeros among $b_1,\dots,b_d$, and let $b:\mathbb{F}^{d}\times\mathbb{F}^{d}\rightarrow \mathbb{F}$ be the bilinear form defined as $b(v,w)=\sum_{i=1}^{d}b_i v(i)w(i)$. Let $V<\mathbb{F}^{d}$ such that $b(v,w)=0$ for every $v,w\in V$. Then $\dim(V)\leq \frac{1}{2}(d+z)$.
	\end{lemma}
	
	\begin{proof}
		Let $M$ be the $d\times d$ diagonal matrix with diagonal entries $b_{1},\dots,b_{d}$, and let $W=\{Mv:v\in V\}<\mathbb{F}^{d}$. Then $\dim(\mbox{ker }M)=z$, so $\dim(W)\geq\dim(V)-z$. 
		By definition, $V$ and $W$ are orthogonal spaces (with respect to the standard inner product). Therefore, $\dim(V)+\dim(W)\leq d$, which implies $\dim(V)\leq \frac{1}{2}(d+z)$.
	\end{proof}
	
	In what comes, we show that if $\ell=p^{\alpha}$ is a prime power, and  $\mathcal{F}\subset \{0,1\}^n$ is $k$-closed over $\mathbb{Z}_{\ell}$ for some large constant $k$, then most sets of maximal twins for $\mathcal{F}$ must have size divisible by $\ell$, provided that the dimension of $\langle \mathcal{F}\rangle_p$ is large. We start with the case when $\ell$ is a prime.
	
	\begin{lemma}\label{lemma:prime}
		Let $V<\mathbb{F}_p^{n}$, let $A_1,\dots,A_d$ be a partition of $[n]$ into twins for $V$, and suppose that $V$ is 2-closed. If $\dim(V) \geq d-h$, then at least $d-2h$ of the numbers $|A_1|,\dots,|A_{d}|$ are divisible by $p$.
	\end{lemma}
	
	\begin{proof}
		For $i\in [d]$, let $b_i=|A_i|$ and let $b$ be the bilinear form defined as in Lemma \ref{lemma:bilinear}. Let $\phi:V\rightarrow\mathbb{F}_p^{d}$ be the linear map defined as $\phi(v)(i)=s$ if $v|_{A_i}$ is the constant $s$ vector ($i = 1,\dots,d$). Then $\phi$ is an injection, so $\dim(\phi(V))=\dim(V)=d-h$. Also, by the definition of $b$, for every $u,v\in V$ we have $||u\cdot v||= \sum_{i = 1}^n{u(i) \cdot v(i)} = \sum_{i=1}^d{|A_i| \cdot \phi(u)(i) \cdot \phi(v)(i)} =  b(\phi(u),\phi(v))$, so we have $b(x,y)=0$ for every $x,y\in\phi(V)$. But then by Lemma \ref{lemma:bilinear}, if $z$ is the number of zeros among $b_1,\dots,b_d$, then $\dim(\phi(V))\leq \frac{1}{2}(d+z)$, which gives $z\geq d-2h$.
	\end{proof}
	
	\begin{lemma}\label{lemma:primepower}
		Let $p$ be a prime and $\alpha\in\mathbb{Z}^{+}$. Let $\mathcal{F}\subset \{0,1\}^{n}$ be $2(p+\alpha)$-closed over $\mathbb{Z}_{p^{\alpha}}$, let $\dim(\langle \mathcal{F}\rangle_p)=d$, and let $A_1,\dots,A_d,B$ be a partition of $[n]$ such that $A_{i}$ is a set of twins for $\mathcal{F}$, and $\dim(\langle \mathcal{F}|_{B}\rangle_{p})\leq h$.  Then at least $d-2\alpha h$ of the numbers $|A_1|,\dots,|A_d|$ are divisible by $p^{\alpha}$.
	\end{lemma}
	
	\begin{proof}
		We prove this by induction on $\alpha$. The case $\alpha=1$ follows from by applying Lemma \ref{lemma:prime} to 
		$V = \langle \mathcal{F}|_{[n] \setminus B}\rangle_{p}$, 
		noting that 
		$\dim(V)\geq \dim(\langle \mathcal{F}\rangle_p) - \dim(\langle \mathcal{F}|_{B}\rangle_{p}) \geq d-h$.
		
		Suppose now that $\alpha>1$. Then, by our induction hypothesis, at least $k=d-(2\alpha-2)h$ of the sets $A_1,\dots,A_d$ have size divisible by $p^{\alpha-1}$, without loss of generality, let these sets be $A_1,\dots,A_k$. Also, let $B'=A_{k+1}\cup\dots\cup A_{d}\cup B$. Letting $V=\langle \mathcal{F}\rangle_p$, note that $$\dim(V|_{B'})\leq \dim(V|_{B})+(d-k)\leq h+d-k.$$ 
		Let $W < V$ be the subspace of all $v \in V$ which vanish on $B'$. Then $\dim(W)\geq d-\dim(V|_{B'})\geq k-h$.
		
		Note that for every $w\in W$ there exists some $w'\in \langle\mathcal{F}\rangle_{p^{\alpha}}$ such that $w'\equiv w \pmod p$. (Indeed, $w \in \langle\mathcal{F}\rangle_{p}$, so $w$ is a linear combination of elements of $\mathcal{F}$ with coefficients in $\mathbb{F}_p$. Taking the same linear combination but over $\mathbb{Z}_{p^{\alpha}}$ gives $w'$.) Let $\beta$ be the smallest number such that $\beta \geq \alpha$ and $\beta=1 \pmod{p-1}$, then $\beta<\alpha+p$. As $\mathcal{F}$ is $2\beta$-closed over $\mathbb{Z}_{p^{\alpha}}$, for every $u,v\in W$ we have that $||(u')^{\beta}\cdot (v')^{\beta}||$ is divisible by $p^{\alpha}$. Note also that we have the following two properties:
		\begin{enumerate}
		    \item[(a)]$(w')^{\beta}\equiv w \pmod p$, as $\beta \equiv 1 \pmod{p-1}$;
		    \item[(b)]If $w(i)=0$ (over $\mathbb{F}_p$) for some $i\in [n]$, then $p^{\alpha}\mid (w')^{\beta}(i)$, because $\beta \geq \alpha$. This means that $(w')^{\beta}(i) \equiv 0 \pmod{p^{\alpha}}$ for every $i \in B'$, since $w$ vanishes on $B'$.
		\end{enumerate}
		For $i\in [k]$, let $A_{i}'$ be a set of size $|A_{i}|/p^{\alpha-1}$, and let $A'=\bigcup_{i=1}^{k}A_{i}'$. Define the linear map $\phi:W\rightarrow\mathbb{F}_p^{A'}$ as follows. If $w\in W$, $i\in [k]$, and $w|_{A_i}$ is the constant $s$ vector, then $\phi(w)|_{A_i'}$ is the constant $s$ vector. 
		(Every $w \in W$ is constant on each $A_i$ because $A_i$ is a set of twins for $V$, and $w$ is determined by its values on $A_1,\dots,A_k$ as $w$ vanishes on $B' = [n] \setminus (A_1 \cup \dots \cup A_k)$.)
		Then $\phi$ is an injection, so $\dim(W)=\dim(\phi(W)) \geq k - h$. Also, for every $u,v\in W$, we have 
		$$||(u')^{\beta}\cdot (v')^{\beta}||\equiv p^{\alpha-1}||\phi(u)\cdot \phi(v)|| \pmod {p^{\alpha}}.$$
		Here we use properties (a)-(b).
		So we see that $||x\cdot y||=0$ (over $\mathbb{F}_p$) for every $x,y\in\phi(W)$, so $\phi(W)$ is $2$-closed. Let $z$ be the number of sets among $A_1',\dots,A_k'$, whose size is divisible by $p$. We can apply Lemma \ref{lemma:prime} to the space $\phi(W)$ and the sets $A'_1,\dots,A'_k$ (recalling that $\dim(\phi(W)) \geq k - h$) to conclude that 
		$z\geq k-2h\geq d-2\alpha h.$
		As $z$ is also the number of sets among $A_1,\dots,A_k$ whose size is divisible by $p^{\alpha}$, this finishes the proof.
	\end{proof}
	
	\begin{lemma}\label{lemma:smalldim}
		Let $p$ be a prime, and $\alpha,t\in\mathbb{Z}^{+}$. Let $\mathcal{F}\subset \{0,1\}^{n}$ such that $\mathcal{F}$ is non-reducible and ${2^{t+1}(p+\alpha)}$-closed over $\mathbb{Z}_{p^{\alpha}}$. Let $A_1,\dots,A_d$ be the unique partition of $[n]$ into maximal sets of twins, and \nolinebreak let $$B=\bigcup_{\substack{i\in [d]\\|A_i|\not\equiv 0\ \mathrm{(mod }\ p^{\alpha})}}A_{i}.$$ Then $\dim(\langle \mathcal{F}|_{B}\rangle_{p})\leq \frac{6n\alpha}{t}$.
	\end{lemma}
	
	\begin{proof}
		Let $\mathcal{F}_0=\mathcal{F}$, and for $i=1,2,\dots,t$, let $\mathcal{F}_{i}=\mathcal{F}_{i-1}\cdot\mathcal{F}_{i-1}$. Note that $\mathcal{F}_{i-1} \subset \mathcal{F}_i$, as it is clear that $\mathcal{F}_{i-1}$ contains only 0-1 vectors. Moreover, $\mathcal{F}_i$ is $2^{t+1-i}(p+\alpha)$-closed over $\mathbb{Z}_{p^{\alpha}}$, which follows by induction and Item 3 in Claim \ref{claim:k-closed}. From this, we will only use that $\mathcal{F}_i$ is at least $2(p+\alpha)$-closed. Finally, $A_1,\dots,A_d$ is also the unique partition of $[n]$ into maximal sets of twins for $\mathcal{F}_i$, see Item 2 in Claim \ref{claim:twins}.
		
		The sequence of positive integers $(\dim(\langle \mathcal{F}_r\rangle_p))_{r=0,\dots,t}$ is monotone increasing (using that $\mathcal{F}_{0}\subset\dots\subset \mathcal{F}_{t})$, and its elements are contained in $[n]$, so by the pigeonhole principle, there exists $0\leq r<t$ such that 
		$$\dim(\langle\mathcal{F}_{r+1} \rangle_{p})\leq \dim(\langle\mathcal{F}_r\rangle_{p})+\frac{n}{t}.$$
		Let $d'=\dim(\langle\mathcal{F}_r\rangle_{p})$. Applying Theorem \ref{cor:stability}, we deduce that $[n]$ can be partitioned into
        $d'+1$ sets, $d'$ of which are maximal sets of twins for $\mathcal{F}_r$, and an additional set $C$ satisfying $\dim(\langle \mathcal{F}_r|_{C}\rangle_p)\leq \frac{2n}{t}$. Each maximal set of twins equals $A_i$ for some $i$, so let us assume without loss of generality that the $d'$ maximal sets of twins are $A_1,\dots,A_{d'}$.
		Now, as $\mathcal{F}_{r}$ is $2(p+\alpha)$-closed, we can apply Lemma \ref{lemma:primepower} with $h =  \frac{2n}{t}$ to conclude that at least $q=d'-\frac{4n\alpha}{t}$ of the numbers $|A_{1}|,\dots,|A_{d'}|$ are divisible by $p^{\alpha}$. Without loss of generality, let $A_1,\dots,A_q$ be the sets of twins whose sizes are divisible by $p^{\alpha}$. Let $D=C\cup A_{q+1}\cup\dots\cup A_{d'}$. Then $B\subset D$, and noting that $\dim(\langle\mathcal{F}_r|_{D\setminus C}\rangle_p)\leq d'-q$, and  $\mathcal{F}\subset \mathcal{F}_r$, we get the chain of inequalities $$\dim(\langle \mathcal{F}|_{B}\rangle_p)\leq \dim(\langle \mathcal{F}_{r}|_{D}\rangle_p)\leq (d'-q)+\dim(\langle \mathcal{F}_r|_{C}\rangle_p)\leq \frac{4n\alpha+2n}{t} \leq \frac{6n\alpha}{t}.$$ 
		This finishes the proof.
	\end{proof}
	
	The final ingredient we need for the proof of Theorem \ref{thm:main} is the following well known result, see e.g. the work of Odlyzko \cite{O81}.
	
	\begin{lemma}\label{lemma:odim}
		Let $p$ be a prime and $V< \mathbb{F}_p^{n}$. Then $|V\cap\{0,1\}^{n}|\leq 2^{\dim(V)}$.
	\end{lemma}
	
	\begin{proof}[Proof of Theorem \ref{thm:main}]
		Write $\ell=p_1^{\alpha_1}\dots p_s^{\alpha_s}$, where $p_1,\dots,p_s$ are distinct primes. We show that choosing $k=2^{t+1}\max_{r\in [s]}(p_r+\alpha_r)$ suffices, where $t=12\ell\sum_{r=1}^{s}\alpha_{r}$. More precisely, we show the following two statements. Let $\mathcal{F}\subset \{0,1\}^{n}$ such that $\mathcal{F}$ is $k$-closed over $\mathbb{Z}_{\ell}$.
		\begin{description}
			\item[(1)] Then $|\mathcal{F}|\leq 2^{\lfloor n/\ell\rfloor}$.
			\item[(2)] If $|\mathcal{F}|>2^{\lfloor n/\ell\rfloor-1}$, then $[n]$ can be partitioned into sets $A_{1},\dots,A_{\lfloor n/\ell\rfloor},A'$ such that $A_{i}$ is a maximal set of twins and $|A_i| = \ell$ for each $1 \leq i \leq \lfloor n/\ell\rfloor$, $|A'|\leq \ell-1$, and $\mathcal{F}$ vanishes on $A'$.
		\end{description}

		We proceed by induction on $n$. If $n\leq 6s\ell$, the statements are easy to show. Indeed, let $A_1,\dots,A_d,A'$ be a partition of $[n]$ such that $A_{i}$ is a maximal set of twins for $\mathcal{F}$ for $i\in [d]$, and $\mathcal{F}$ vanishes on $A'$. Then $|\mathcal{F}|\leq 2^{d}$. We argue that as $k>2^{t}\geq 2^{12\ell s}> 6s\ell\geq n$, the  characteristic vector of each $A_{i}$ is contained in $(\langle \mathcal{F} \rangle_{\ell})^{k}$. So fix any $1 \leq i \leq d$. First of all, note that there exists  $v_i\in\mathcal{F}$ such that $v_{i}|_{A_i}$ is the all 1 vector. Now for $j\in [d]\setminus \{i\}$, define $v_j$ as follows. 
		If $v_{i}|_{A_j}$ is 0, then let $v_j=v_i$. Otherwise, as the elements of $A_i$ are not twins of the elements of $A_j$, there exists $w_{j}\in\mathcal{F}$ such that one of $w_j|_{A_{i}}$ and $w_{j}|_{A_{j}}$ is the all 1 vector, and the other is 0. If $w_{j}|_{A_j}$ is 0, then set $v_j=w_j$, otherwise set $v_j=v_i-w_j$. In all cases, $v_j$ has the property that $v_{j}|_{A_i}$ is the all 1 vector, and $v_{j}|_{A_j}$ is 0. But then $v_1\cdot \dotsc \cdot v_{d}$ is the characteristic vector of $A_i$, and as $d\leq n<k$, we have $v_1\cdot\dotsc\cdot v_d\in (\langle \mathcal{F} \rangle_{\ell})^{k}$. Now, since $\mathcal{F}$ is $k$-closed, this implies that $\ell$ divides $|A_{i}|$ for $i\in [d]$. But then $d\leq \lfloor n/\ell\rfloor$, and we are done with (1). Also, if $|\mathcal{F}|>2^{\lfloor n/\ell\rfloor-1}$, we must have $d=\lfloor n/\ell\rfloor$, which is only possible if all the sets $A_1,\dots,A_d$ have size $\ell$. Therefore, (2) also holds.
		
	    From now on assume that $n>6s\ell$. First, suppose that there exists $A\subset [n]$ such that $\ell$ divides $|A|$ and $A$ is a set of twins for $\mathcal{F}$. Then the family $\mathcal{F}'=\mathcal{F}|_{[n]\setminus A}$ is also $k$-closed over $\mathbb{Z}_{\ell}$ and $|\mathcal{F}'|\geq \frac{1}{2}|\mathcal{F}|$. By our induction hypothesis, we have $|\mathcal{F}'|\leq 2^{\lfloor (n-|A|)/\ell\rfloor}\leq 2^{\lfloor (n-\ell)/\ell\rfloor}$, so we get $|\mathcal{F}|\leq 2^{\lfloor n/\ell\rfloor}$, and (1) indeed holds. If $|\mathcal{F}|> 2^{\lfloor n/\ell\rfloor-1}$, then $|\mathcal{F}'|>2^{\lfloor (n-\ell)/\ell\rfloor-1}$, so by our induction hypothesis there exists a partition of $[n]\setminus A$ into sets $A_{1},\dots,A_{\lfloor (n-\ell)/\ell\rfloor},A'$ satisfying (2) with respect to $\mathcal{F}'$. Setting $A_{\lfloor n/\ell\rfloor}=A$, the sets  $A_{1},\dots,A_{\lfloor n/\ell\rfloor},A'$ satisfy (2) with respect to $\mathcal{F}$.
		
		Therefore, in order to finish the proof, it is enough to show that if $|\mathcal{F}|>2^{\lfloor n/\ell\rfloor-1}$, then $\mathcal{F}$ has a set of twins of size divisible by $\ell$. Next, we show that if $I\subset [n]$ is large, then the dimension of $\langle \mathcal{F}|_{I}\rangle_p$ cannot be too small for any prime $p$.
		
		\begin{claim}\label{claim:coordinates}
			Let $p$ be a prime and $I\subset [n]$ such that $|I|\geq \ell$. Then 
			$$|I|\leq \ell\dim(\langle \mathcal{F}|_I\rangle_{p})+3\ell.$$
		\end{claim}
		
		\begin{proof}
			Let $V=\langle \mathcal{F}|_{I}\rangle_p$ and $d = \dim(V)$. Then $|V\cap \{0,1\}^{I}|\leq 2^{d}$ by Lemma \ref{lemma:odim}. By the pigeonhole principle, there exists some $v\in \{0,1\}^{I}$ and $\mathcal{F}'\subset \mathcal{F}$ such that $w|_{I}=v$ for every $w\in\mathcal{F}'$, and $|\mathcal{F}'|\geq |\mathcal{F}|/2^{d}$. Let $0\leq m\leq \ell-1$ such that $||v||\equiv m \pmod \ell$, and for every $w\in \mathcal{F}'$, let $w^{*}$ be the vector we get after replacing the coordinates in $I$ with $m$ coordinates of $1$ entries. This gives a family $\mathcal{F}''=\{w^{*}:w\in\mathcal{F}'\}\subset \{0,1\}^{n-|I|+m}$ such that $|\mathcal{F}''|=|\mathcal{F}'|\geq |\mathcal{F}|/2^{d}$ and $\mathcal{F}''$ is $k$-closed over $\mathbb{Z}_{\ell}$. The latter is true because if $w_1,\dots,w_k\in\mathcal{F}'$, \nolinebreak then  
			$$||w_1^{*}\cdot\dotsc \cdot w_{k}^{*}||=||(w_1|_{[n]\setminus I})\cdot\dotsc\cdot (w_{k}|_{[n]\setminus I})||+m\equiv||w_1\cdot\dotsc\cdot w_{k}||\equiv 0 \pmod \ell .$$
			Therefore, by our induction hypothesis, we have
			$$2^{\lfloor n/\ell\rfloor-d-1}<\frac{|\mathcal{F}|}{2^{d}}\leq |\mathcal{F}''|\leq 2^{\lfloor (n-|I|+m)/\ell\rfloor}< 2^{\lfloor n/\ell\rfloor+2-|I|/\ell}.$$
			Comparing the left- and right-hand-side gives the desired inequality $|I|\leq \ell d+3\ell$. 
		\end{proof}	
		
		We continue with the proof of the theorem.
		We can assume that $\mathcal{F}$ is non-reducible, otherwise we are immediately done 
		by applying our induction hypothesis. Let $A_{1},\dots,A_{d}$ be the unique partition of $[n]$ such that $A_{i}$ is a maximal set of twins for $\mathcal{F}$. Let $r\in [s]$, and apply Lemma \ref{lemma:smalldim} to $\mathcal{F}$ with respect to the prime power $p_r^{\alpha_r}$. Let $$B_r=\bigcup_{\substack{i\in [d]\\|A_i|\not\equiv 0\ \mathrm{(mod }\ p_r^{\alpha_r})}}A_{i}.$$  As $\mathcal{F}$ is $2^{t+1}(p_r+\alpha_r)$-closed, we get from Lemma \ref{lemma:smalldim} that $\dim(\langle \mathcal{F}|_{B_r}\rangle_{p_r})\leq \frac{6n\alpha_r}{t}$. But then by Claim \ref{claim:coordinates}, we also have 
		$$|B_r|\leq \frac{6n\alpha_r\ell}{t}+3\ell.$$
		Let $B=\bigcup_{r=1}^{s}B_{r}$, then 
		$$|B|\leq\sum_{r=1}^{s}|B_{r}|\leq 3s\ell+\frac{6n\ell}{t}\sum_{r=1}^{s}\alpha_r <n,$$
		where the last inequality holds by the choice of $t$ and noting that $n > 6sl$. Observe that $B$ is the union of those maximal sets of twins $A_{i}$ where $|A_i|$ is not divisible by $\ell$. Therefore, as $|B|< n$ and $A_{1},\dots,A_{d}$ form a partition of $[n]$, there must exist $i\in [d]$ such that $\ell$ divides $|A_{i}|$, finishing the proof.
	\end{proof}
	
	Let us remark that in case $\ell$ is a prime power, we can prove something slightly stronger following the same proof. This might be of independent interest.
	
	\begin{theorem}\label{thm:dimcor}
		Let $p$ be a prime, $\ell=p^{\alpha}$, then there exists $k$ such that the following holds. Let $\mathcal{F}\subset \{0,1\}^{n}$ such that $\mathcal{F}$ is $k$-closed over $\mathbb{Z}_{\ell}$. Then $\dim(\langle \mathcal{F}\rangle_{p})\leq \lfloor n/\ell\rfloor$.
	\end{theorem}
	
	\begin{proof}
		We show that $k=2^{t+1}(p+\alpha)$ suffices, where $t=12\alpha\ell$. We will proceed by induction on $n$. In case $n\leq \ell$, the statement is trivial, so assume that $n>\ell$. It suffices to prove the theorem for $\mathcal{F}$ which is maximal $k$-closed over $\mathbb{Z}_{\ell}$, namely that we cannot add elements to $\mathcal{F}$ and keep it $k$-closed. So assume that $\mathcal{F}$ is maximal in this sense. 
		Let $V=\langle \mathcal{F}\rangle_p$.
		
		Suppose that there exists a set $A\subset [n]$ of twins of size $\ell$ for $\mathcal{F}$. If $v\in \{0,1\}^{n}$ is the characteristic vector of $A$, then $v\in \mathcal{F}$, otherwise $\{v\}\cup \mathcal{F}$ contradicts the maximality of $\mathcal{F}$.  Let $\mathcal{F}'=\mathcal{F}|_{[n]\setminus A}$ and $V'=\langle \mathcal{F}'\rangle_p$, then $\dim(V')=\dim(V)-1$ (as $v$ is an element of $V$) and $\mathcal{F}'$ is also $k$-closed over $\mathbb{Z}_{\ell}$. Therefore, by our induction hypothesis, $\dim(V')\leq \lfloor (n-\ell)/\ell\rfloor$, so $\dim(V)\leq \lfloor n/\ell\rfloor$.
		
		In the rest of the proof, we show that if $\dim(V)\geq \lfloor n/\ell\rfloor$, then there exists a set $A\subset [n]$ of twins of size $\ell$ for $\mathcal{F}$. We can assume that $\mathcal{F}$ is non-reducible, otherwise apply our induction hypothesis. Let $A_{1},\dots,A_{d}$ be the unique partition of $[n]$ such that $A_{i}$ is a maximal set of twins for $\mathcal{F}$. Apply Lemma \ref{lemma:smalldim} to $\mathcal{F}$ with respect to the prime power $p^{\alpha}$. Let $$B=\bigcup_{\substack{i\in [d]\\|A_i|\not\equiv 0\ \mathrm{(mod }\ p^{\alpha})}}A_{i}.$$  As $\mathcal{F}$ is $2^{t+1}(p+\alpha)$-closed, we get that 
		$$\dim(V|_{B})\leq \frac{6n\alpha}{t}< \left\lfloor \frac{n}{\ell}\right\rfloor \leq \dim(V).$$
		Therefore, $B\neq [n]$, and at least one of $|A_{1}|,\dots,|A_{d}|$ is divisible by $p^{\alpha}$. Hence, $\mathcal{F}$ has a set of twins of size $\ell = p^{\alpha}$. This finishes the proof.
	\end{proof}
	
	Finally, we note that if $\ell$ is a prime, a similar proof shows that Theorem \ref{thm:dimcor} holds for every $\mathcal{F}\subset \mathbb{F}_{\ell}^{n}$ (that is, the elements of $\mathcal{F}$ need not be 0-1 vectors).
	
	\begin{theorem}
		Let $p$ be a prime, then there exists $k$ such that the following holds. Let $\mathcal{F}\subset \mathbb{F}_p^{n}$ such that $\mathcal{F}$ is $k$-closed. Then $\dim(\langle \mathcal{F}\rangle_{p})\leq \lfloor n/p\rfloor$. In particular, $|\mathcal{F}|\leq p^{\lfloor n/p\rfloor}$, and this bound is the best possible.
	\end{theorem}
	
	\section{Proof of the main result}
	
	In this section, we show how to deduce Theorem \ref{thm:main0} from Theorem \ref{thm:main}. Let us start with the following variant of the well known Oddtown theorem \cite{B69}, see also \cite{BF92} for related results.
	
	\begin{lemma}\label{lemma:oddtown}
		Let $\ell,m,n$ be positive integers, and let $A_1,\dots,A_m,B_1,\dots,B_m\subset [n]$ such that $\ell\nmid|A_{i}\cap B_{i}|$ for $i\in [m]$, but $\ell$ divides $|A_{i}\cap B_{j}|$ for $i\neq j$. Then $m\leq sn$, where $s$ is the number of distinct prime divisors of~$\ell$.
	\end{lemma}
	
	\begin{proof}
		Write $\ell=p_1^{\alpha_1}\dots p_{s}^{\alpha_s}$, where $p_1,\dots,p_s$ are distinct primes. Let $v_{i}$ and $w_{i}$ be the characteristic vectors of $A_{i}$ and $B_{i}$ over $\mathbb{Q}$, respectively. Let $t=\lceil m/s\rceil$, then there exists $r\in [s]$ such that for at least $t$ of the indices $i\in [m]$, we have that $|A_{i}\cap B_{i}|$ is not divisible by $p_{r}^{\alpha_r}$. Without loss of generality let these $t$ indices be $1,\dots,t$. We show that $v_1,\dots,v_t$ are linearly independent (over $\mathbb{Q}$), which then implies $t\leq n$ and $m\leq sn$.
		
		Suppose this is not the case, then there exist $c_1,\dots,c_t\in\mathbb{Z}$, not all zero, such that $\sum_{i=1}^{t}c_{i}v_{i}=0$. We can assume that at least one of $c_1,\dots,c_t$ is not divisible by $p_r$, otherwise we can replace $c_i$ with $c_i'=c_i/p_r$ for every $i\in [t]$. Let $k\in [t]$ be an index such that $p_r\nmid c_{k}$. Consider the equality 
		$$0=\left\langle\sum_{i=1}^{t}c_{i} v_{i},w_{k}\right\rangle=\sum_{i=1}^{t}c_{i} |A_i\cap B_k|.$$
		We have $p_r^{\alpha_{r}}\mid c_i|A_{i}\cap B_{k}|$ if $i\neq k$, and $p_r^{\alpha_r}\nmid c_k|A_{k}\cap B_{k}|$, so $p_r^{\alpha_{r}}\nmid \langle\sum_{i=1}^{t}c_{i} v_{i},w_{k}\rangle$, contradiction.
	\end{proof}
	
	Say that $\mathcal{F}\subset 2^{[n]}$ is \emph{weakly $k$-closed over $\mathbb{Z}_{\ell}$} if the intersection of any $k$ distinct elements of $\mathcal{F}$ is divisible by $\ell$. Also, say that $\mathcal{F}\subset 2^{[n]}$ is \emph{$k$-closed over $\mathbb{Z}_{\ell}$} if the family formed by the characteristic vectors of the elements of $\mathcal{F}$ is $k$-closed over $\mathbb{Z}_{\ell}$. So $\mathcal{F}$ is $k$-closed over $\mathbb{Z}_{\ell}$ if and only if the intersection of any $k$ not necessarily distinct elements of $\mathcal{F}$ is divisible by $\ell$. We need the following useful observation, which for $\ell=2$ appears in \cite{SV18}. 
	
	\begin{lemma}\label{lemma:weakly}
		Let $\ell,k$ be positive integers, and let $s$ be the number of distinct prime divisors of $\ell$. Let $\mathcal{F}\subset 2^{[n]}$ such that $\mathcal{F}$ is weakly $k$-closed over $\mathbb{Z}_{\ell}$.  Then there exists $\mathcal{F}'\subset \mathcal{F}$ such that $|\mathcal{F}'|\geq |\mathcal{F}|-sk^2n$, and $\mathcal{F}'$ is $k$-closed over $\mathbb{Z}_{\ell}$. 
	\end{lemma}
	
	\begin{proof}
		Repeat the following removal operation. Suppose that $\mathcal{F}$ is not $k$-closed, and let $t$ be maximal such that some $t$ distinct elements of $\mathcal{F}$ have an intersection not divisible by $\ell$. So $t < k$ because $\mathcal{F}$ is weakly $k$-closed. Let $\mathcal{H}$ be the $t$-uniform hypergraph on $\mathcal{F}$ in which $\{C_{1},\dots,C_{t}\}$ is an edge if $|C_1\cap\dots\cap C_t|$ is not divisible by $\ell$. We claim that $\mathcal{H}$ contains no matching of size more than $sn$. Indeed, suppose otherwise, let $\{C_{i,1},\dots,C_{i,t}\}$, $i\in [m]$, be the edges of a matching of size $m>sn$. For $i\in [m]$, let $A_{i}=C_{i,1}$ and $B_{i}=C_{i,1}\cap\dots\cap C_{i,t}$. Then $\ell\nmid|A_{i}\cap B_{i}|$, but $\ell$ divides $|A_i\cap B_j|$ for every $i\neq j$ by the maximality of $t$. Therefore, by Lemma \ref{lemma:oddtown} we get $m\leq sn$, contradiction.
		
		Consider a maximal matching of $\mathcal{H}$, and remove every element of $\mathcal{F}$ that appears in this matching. Then $\mathcal{F}$ no longer contains $t$ distinct sets, whose intersection is not divisible by $\ell$. Repeating this procedure at most $k-1$ times, we get a family $\mathcal{F}'\subset \mathcal{F}$ such that $\mathcal{F}'$ is $k$-closed over $\mathbb{Z}_{\ell}$, and $|\mathcal{F}'|\geq |\mathcal{F}|-sk^2n$.
	\end{proof}
	
	Lemma \ref{lemma:weakly} combined with Theorem \ref{thm:main} immediately implies that if $\mathcal{F}\subset 2^{[n]}$ is weakly $k$-closed over $\mathbb{Z}_{\ell}$, then $|\mathcal{F}|\leq 2^{\lfloor n/\ell\rfloor}+sk^2n$. In order to improve the term $sk^2n$ to a constant, we use the second part of Theorem~\ref{thm:main}.
	
	\begin{proof}[Proof of Theorem \ref{thm:main0}]
		Let $d=\lfloor n/\ell\rfloor$. Let $\mathcal{F}\subset\{0,1\}^{n}$ such that $\mathcal{F}$ is weakly $k$-closed over $\mathbb{Z}_{\ell}$. Then by Lemma \ref{lemma:weakly}, there exists $\mathcal{F}'\subset \mathcal{F}$ such that $|\mathcal{F}'|\geq |\mathcal{F}|-sk^2n$ and $\mathcal{F}'$ is $k$-closed over $\mathbb{Z}_{\ell}$, where $s$ is the number of distinct prime divisors of $\ell$. If $|\mathcal{F}'|\leq 2^{d-1}$, then 
		$$|\mathcal{F}|\leq 2^{d-1}+sk^2n<2^{d},$$
		where the last inequality holds if $n$ is sufficiently large. 
		
		Suppose that $|\mathcal{F}'|> 2^{d-1}$ and $|\mathcal{F}|\geq 2^{d}$, otherwise we are done. Then by Theorem \ref{thm:main}, $[n]$ can be partitioned into sets $A_{1},\dots,A_{d},A'$ such that $A_{i}$ is a maximal set of twins for $\mathcal{F}'$ for $i\in [d]$, $|A_{i}|=\ell$, $|A'|\leq \ell-1$, and $\mathcal{F}'$ vanishes on $A'$. Let $S\subset \{0,1\}^{n}$ be the atomic family containing all possible $2^{d}$ sets $C$ such that $C\cap A_{i}\in \{\emptyset,A_{i}\}$ for every $i\in [d]$. Then $\mathcal{F}'\subset S$ and $|S\setminus\mathcal{F}'|\leq sk^2n$, as $|\mathcal{F}'| \geq 2^d - sk^2n$. Also, if $n$ is sufficiently large, for every $i\in [d]$ we can find $k-1$ distinct sets $B_{i,1},\dots,B_{i,k-1}\in \mathcal{F}'$ such that $A_i=\bigcap_{j=1}^{k-1}B_{i,j}$. Indeed, $\mathcal{F}'$ contains a set of the form $A_{i}\cup A_{a}\cup A_{b}$ for some $a,b$, as the number of such sets in $S$ is $\binom{d-1}{2}>sk^2n$. Let $B_{i,1}$ be such a set $A_{i}\cup A_{a}\cup A_{b} \in \mathcal{F}'$. Also $\mathcal{F}'$ contains $k-2$ sets that contain $A_i$ but do not contain $A_{a}$ and $A_b$ as the number of such sets in $S$ is $2^{d-3}> sk^2n+k$. Let these $k-2$ sets be $B_{i,2},\dots,B_{i,k-1}$. Then $A_{i}=\bigcap_{j=1}^{k-1}B_{i,j}$, as claimed. 
		
		Consider any $F\in \mathcal{F}\setminus S$.
		For every $i\in [d]$, we have $A_i\subset F$ or $A_i\cap F=\emptyset$, as the size of  $A_i\cap F=B_{i,1}\cap\dots\cap B_{i,k-1}\cap F$ must be divisible by $\ell$. 
		Now, as $F \notin S$, we must have 
		$F \cap A' \neq \emptyset$.
		On the other hand, for any $H\subset A'$ with $H\neq \emptyset$, there are at most $k-1$ elements $F\in\mathcal{F}\setminus S$ such that $F\cap A'=H$, because otherwise we would have $k$ distinct $F_1,\dots,F_k \in \mathcal{F}$ with $|F_1 \cap \dots \cap F_k| \equiv |H| \not\equiv 0 \pmod{\ell}$, a contradiction. 
		So we see that $|\mathcal{F}\setminus S|\leq k2^{|A'|}\leq k2^{\ell-1}$, and if $\ell \mid n$ then $\mathcal{F}\subset S$ because $A' = \emptyset$. This finishes the proof.
	\end{proof}
	
	\section{Stability}
	
	In this section, we prove Theorem \ref{thm:stab}. The proof follows easily from our earlier results.
	
	\begin{proof}[Proof of Theorem \ref{thm:stab}]
		Write $\ell=p_1^{\alpha_1}\dots p_{s}^{\alpha_s}$, where $p_1,\dots,p_s$ are distinct primes. Let $t=\lceil 6\epsilon^{-1}\sum_{r\in [s]}\alpha_r\rceil$, then we show that $k=2^{t+1}\max_{r\in [s]} (p_r+\alpha_r)$ suffices.

		Without loss of generality, $\mathcal{F}$ is non-reducible. Let $A_1,\dots,A_d$ be the unique partition of $[n]$ into maximal sets of twins for $\mathcal{F}$. For $r\in [s]$, let 
		$$B_r=\bigcup_{\substack{i\in [d]\\ |A_i|\not\equiv 0 \mathrm{(mod }\ p_r^{\alpha_r})}} A_{i}.$$
		Then by Lemma \ref{lemma:smalldim}, we have $\dim(\langle \mathcal{F}|_{B_r}\rangle_{p_r})\leq \frac{6n\alpha_r}{t}.$
		This implies, using Lemma \ref{lemma:odim}, that $$|\mathcal{F}|_{B_r}|\leq 2^{6n\alpha_r/t}.$$
		Let $B=\bigcup_{r\in [s]} B_{r}$. Then $B$ is the union of those sets $A_i$, whose size is not divisible by $\ell$. Also, we have
		$$|\mathcal{F}|_{B}|\leq \prod_{r\in [s]} |\mathcal{F}|_{B_r}|\leq 2^{6n(\alpha_1+\dots+\alpha_s)/t}\leq 2^{\epsilon n}.$$
		Set $X=[n]\setminus B$, then $|\mathcal{F}|_{X}|\geq2^{-\epsilon n}|\mathcal{F}|$. But $X$ is the union of sets of twins of size divisible by $\ell$, so we can partition $X$ into sets of twins of size exactly $\ell$. Hence, $\mathcal{F}|_{X}$ is a subfamily of some isomorphic copy of $S(|X|,\ell)$, finishing the proof.
	\end{proof}
	
	\section{Concluding remarks}
	
	For many problems in extremal set theory, it is natural to consider their {\em multipartite} (i.e. {\em cross}) variant. For example, Frankl and Kupavskii \cite{FK19} proved a tight bound on $|\mathcal{F}_1| + \dots +|\mathcal{F}_k|$ for $k$ set-families $\mathcal{F}_1,\dots,\mathcal{F}_k$ with no choice of $k$ disjoint sets $F_1,\dots,F_k$, $F_i \in \mathcal{F}_i$. In a similar vein, Buci\'{c}, Letzter, Sudakov and Tran \cite{BLST18} proved a tight result for the multipartite version of the Erd\H{o}s--Kleitman conjecture. We refer the reader to the book \cite{FT_Book} for additional examples.
	
	One might wonder what can be said about the cross-version of Theorem \ref{thm:main0}. That is, let $\mathcal{F}_1,\dots,\mathcal{F}_{k}\subset 2^{[n]}$ such that $|F_1\cap\dots\cap F_{k}|$ is divisible by $\ell$ for every $F_{1}\in\mathcal{F}_1,\dots,F_{k}\in\mathcal{F}_{k}$. What is the maximum of $|\mathcal{F}_1|\dots|\mathcal{F}_{k}|$? It follows from Theorem \ref{thm:main} that if $\mathcal{F}_{1}=\dots=\mathcal{F}_{k}$, then this maximum is $2^{k\lfloor n/\ell\rfloor}$, given $k$ is sufficiently large with respect to $\ell$. However, if the families $\mathcal{F}_1,\dots,\mathcal{F}_k$ are not necessarily equal, the answer is very different. Indeed, let $A_{1},\dots,A_{k}$ be an arbitrary partition of $[n]$, and let $\mathcal{F}_{i}=2^{[n]\setminus A_{i}}$ for $i\in [k]$. Then $F_{1}\cap\dots\cap F_{k}=\emptyset$ for every $F_{i}\in \mathcal{F}_{i}$, $i\in [k]$, and $|\mathcal{F}_{1}|\dots|\mathcal{F}_{k}|=2^{(k-1)n}$. We can also show that $2^{(k-1)n}$ is the maximum, which somewhat surprisingly does not depend on $\ell$.
	
	\begin{theorem}
		Let $\ell,k\geq 2$, and let $\mathcal{F}_1,\dots,\mathcal{F}_{k}\subset 2^{[n]}$ such that $\ell$ divides $|F_1\cap\dots\cap F_{k}|$ for every $F_{1}\in\mathcal{F}_1,\dots,F_{k}\in\mathcal{F}_{k}$. Then $|\mathcal{F}_{1}|\dots|\mathcal{F}_{k}|\leq 2^{(k-1)n}$.
	\end{theorem}
	
	\begin{proof}
		Let $p$ be any prime divisor of $\ell$, and let $V_{i}=\langle \mathcal{F}_{i}\rangle_p$ for $i\in [k]$ (where we identify the members of $\mathcal{F}_i$ with their characteristic vectors). Note that for every $v\in V_1\cdot...\cdot V_{k}$, we have $||v||=0$. Our goal is to show the inequality $$\sum_{i=1}^{k}\dim(V_{i})\leq (k-1)n,$$ which then immediately implies the desired bound $|\mathcal{F}_{1}|\dots|\mathcal{F}_{k}|\leq 2^{(k-1)n}$ noting that $|\mathcal{F}_{i}|\leq 2^{\dim(V_{i})}$.
		
		Let $V=V_1\cap\dots\cap V_{k-1}$, then $\dim(V)\geq \sum_{i=1}^{k-1}\dim(V_i)-(k-2)n$. Also, $V^{k-1}$ and $V_{k}$ are orthogonal spaces, so $\dim(V^{k-1})+\dim(V_{k})\leq n$. Therefore, in order to finish the proof, it is enough to show that $\dim(V^{k-1})\geq \dim(V)$, as then
		$$\sum_{i=1}^{k}\dim(V_{i})\leq \dim(V)+(k-2)n+\dim(V_{k})\leq \dim(V^{k-1})+\dim(V_{k})+(k-2)n\leq (k-1)n.$$
		Let $d=\dim(V)$, then there exist $1\leq i_1<\dots<i_d\leq n$ and a basis $v_{1},\dots,v_{d}\in V$ such that $v_{r}(i_{r})=1$ and $v_{r}(i_j)=0$ for $r\in [d]$, $j\in [d]\setminus\{r\}$. But then $v_{r}^{k-1}\in V^{k-1}$, and $v_{r}^{k-1}(i_j)=v_{r}(i_j)$ for $j\in [d]$. Therefore, the vectors $v_{1}^{k-1},\dots,v_{d}^{k-1}$ are linearly independent in $V^{k-1}$, hence $\dim(V^{k-1})\geq d$.
	\end{proof}






\section*{Acknowledgments} 
The authors would like to thank the anonymous referee for their useful comments and suggestions. 
All authors were supported by the SNSF grant 200021\_196965. Istv\'an Tomon also acknowledges the support of Russian Government in the framework of MegaGrant no 075-15-2019-1926, and the support of MIPT Moscow.

\bibliographystyle{amsplain}

\begin{thebibliography}{99}

\bibitem{BF92}
L. Babai and P. Frankl.
\newblock Linear Algebra Methods in Combinatorics. (With Applications to Geometry and Computer Science).
\newblock Preliminary Version, Department of Computer Science, The University of Chicago, 1992. (Version 2.1, 2020)

\bibitem{B69}
E. R. Berlekamp.
\newblock On subsets with intersections of even cardinality."
\newblock {\em Canad. Math. Bull.} 12:471--474, 1969. 

\bibitem{BGT12}
\newblock E. Breuillard, B. Green and T. Tao. 
\newblock The structure of approximate groups."
\newblock {\em Publications math\'ematiques de l'IH\'{E}S}, 116(1):115--221, 2012.

\bibitem{BGT_Survey}
E. Breuillard, B. Green and T. Tao.
\newblock Small doubling in groups.
\newblock In Erd\H{o}s Centennial (pp. 129-151). Springer, Berlin, Heidelberg, 2013. 

\bibitem{BLST18}
M. Buci\'{c}, S. Letzter, B. Sudakov and T. Tran. 
\newblock Minimum saturated families of sets. 
\newblock {\em Bull. Lond. Math. Soc.}, 50(4): 725--732, 2018.

\bibitem{DEF78}
M. Deza, P. Erd\H{o}s and P.Frankl.
\newblock Intersection properties of systems of finite sets.
\newblock {\em Proc. Lond. Math. Soc.} 36:369--384, 1978.

\bibitem{FK19}
P. Frankl and A. Kupavskii. 
\newblock Two problems on matchings in set families --- In the footsteps of Erd\H{o}s and Kleitman. 
\newblock {\em J. Combin. Theory, Ser. B}, 138:286--313, 2019. 

\bibitem{FO83}
P. Frankl and A. M. Odlyzko. 
\newblock On subsets with cardinalities of intersections divisible by a fixed integer.
\newblock {\em European Journal of Combinatorics}, 4(3):215--220, 1983.

\bibitem{FT16}
P. Frankl and N. Tokushige.
\newblock Uniform eventown problems.
\newblock {\em European Journal of Combinatorics}, 51:280--286, 2016.

\bibitem{FT_Book}
P. Frankl and N. Tokushige.
\newblock Extremal problems for finite sets.
\newblock (Vol. 86). American Mathematical Society, 2018.

\bibitem{Freiman}
G. A. Freiman. 
\newblock Foundations of a structural theory of set addition. 
\newblock American Mathematical Society, Providence, R. I., 1973. Translated from the Russian, Translations of Mathematical Monographs, Vol 37.

\bibitem{FS}
Z. F\"uredi and B. Sudakov.
\newblock Extremal set-systems with restricted k-wise intersections.
\newblock {\em J. Combin. Theory, Ser. A} 105:143--159, 2004.

\bibitem{G75}
J. E. Graver.
\newblock Boolean designs and self-dual matroids.
\newblock {\em Linear Algebra Appl.} 10:111--128, 1975.

\bibitem{GR07}
B. Green and I. Z. Ruzsa.
\newblock Freiman's theorem in an arbitrary abelian group.
\newblock {\em J. Lond. Math. Soc.} (2), 75(1):163--175, 2007.

\bibitem{GS02}
V. Grolmusz and B. Sudakov.
\newblock On $k$-wise set-intersections and $k$-wise hamming-distances.
\newblock {\em J. Combin. Theory, Ser. A}, 99:180--190, 2002. 

\bibitem{MST}
D. Munh\'a Correia, B. Sudakov, and I. Tomon.
\newblock Flattening rank and its combinatorial applications.
\newblock {\em Linear Algebra and its Applications}, 625:113-125, 2021.

\bibitem{O81}
A. M. Odlyzko.
\newblock On the ranks of some $(0,1)$-matrices with constant row sums.
\newblock {\em J. Austral. Math. Soc.}, 31:193--201, 1981.

\bibitem{Ruzsa}
I. Z. Ruzsa.
\newblock Sums of finite sets.
\newblock In Number theory (New York, 1991--1995), pages 281--293. Springer, New York, 1996.

\bibitem{SV18}
B. Sudakov and P. Vieira.
\newblock Two remarks on even and oddtown problems.
\newblock {\em SIAM J. Discrete Math.}, 32:280--295, 2018.

\bibitem{SV}
T. Szab\'o and V. H. Vu.
\newblock Exact $k$-wise intersection theorems. 
\newblock {\em Graphs Combin.}, 21:247--261, 2005. 

\bibitem{TV}
T. Tao and V. H. Vu.
\newblock Additive Combinatorics.
\newblock (Vol. 105). Cambridge University Press, 2006. 

\bibitem{V99}
V. H. Vu.
\newblock Extremal set systems with weakly restricted intersections.
\newblock {\em Combinatorica}, 19:567--587, 1999.

\end{thebibliography}


\begin{dajauthors}
\begin{authorinfo}[lg]
  Lior Gishboliner\\
  ETH Z\"urich\\
  Z\"urich, Switzerland\\
  lior.gishboliner\imageat{}math\imagedot{}ethz\imagedot{}ch \\
\end{authorinfo}
\begin{authorinfo}[bs]
  Benny Sudakov\\
  Professor\\
  ETH Z\"urich\\
  Z\"urich, Switzerland\\
  benjamin.sudakov\imageat{}math\imagedot{}ethz\imagedot{}ch \\
\end{authorinfo}
\begin{authorinfo}[it]
  Istv\'an Tomon\\
  ETH Z\"urich\\
  Z\"urich, Switzerland\\
  istvan.tomon\imageat{}math\imagedot{}ethz\imagedot{}ch \\
\end{authorinfo}
\end{dajauthors}

\end{document}